\documentclass[11pt]{amsart}
\input{preamble}

\let\oldeqref\eqref
\makeatletter
\RenewDocumentCommand\eqref{s m}{%
  \IfBooleanTF#1%
  {\textup{\tagform@{\ref*{#2}}}}
  {\oldeqref{#2}}
}
\makeatother

\newcommand{\Bir}{{\operatorname{Bir}}}

\newcommand{\sC}{{\mathcal C}}

\newcommand{\sL}{{\mathcal L}}
\newcommand{\sM}{{\mathcal M}}

\newcommand{\sR}{{\mathcal R}}
\newcommand{\sS}{{\mathcal S}}

\newcommand{\sX}{{\mathcal X}}

\newcommand{\C}{{\mathbb C}}

\newcommand{\N}{{\mathbb N}}

\newcommand{\Q}{{\mathbb Q}}

\newcommand{\Z}{{\mathbb Z}}

\newcommand{\NS}{\operatorname{NS}}

\newcommand{\Pic}{\operatorname{Pic}}

\newcommand{\rk}{{\rm rk}}

\renewcommand{\to}[1][]{\xrightarrow{\ #1\ }}

\makeatletter
\newcommand*{\da@rightarrow}{\mathchar"0\hexnumber@\symAMSa 4B }
\newcommand*{\da@leftarrow}{\mathchar"0\hexnumber@\symAMSa 4C }
\newcommand*{\xdashrightarrow}[2][]{%
  \mathrel{%
    \mathpalette{\da@xarrow{#1}{#2}{}\da@rightarrow{\,}{}}{}%
  }%
}
\newcommand{\xdashleftarrow}[2][]{%
  \mathrel{%
    \mathpalette{\da@xarrow{#1}{#2}\da@leftarrow{}{}{\,}}{}%
  }%
}
\newcommand*{\da@xarrow}[7]{%
  \sbox0{$\ifx#7\scriptstyle\scriptscriptstyle\else\scriptstyle\fi#5#1#6\m@th$}%
  \sbox2{$\ifx#7\scriptstyle\scriptscriptstyle\else\scriptstyle\fi#5#2#6\m@th$}%
  \sbox4{$#7\dabar@\m@th$}%
  \dimen@=\wd0 %
  \ifdim\wd2 >\dimen@
    \dimen@=\wd2 %
  \fi
  \count@=2 %
  \def\da@bars{\dabar@\dabar@}%
  \@whiledim\count@\wd4<\dimen@\do{%
    \advance\count@\@ne
    \expandafter\def\expandafter\da@bars\expandafter{%
      \da@bars
      \dabar@ 
    }%
  }%
  \mathrel{#3}%
  \mathrel{%
    \mathop{\da@bars}\limits
    \ifx\\#1\\%
    \else
      _{\copy0}%
    \fi
    \ifx\\#2\\%
    \else
      ^{\copy2}%
    \fi
  }%
  \mathrel{#4}%
}
\makeatother

\newtheoremstyle{citing}
  {}
  {}
  {\itshape}
  {}
  {\bfseries}
  {\textbf{.}}
  {.5em}
  {\thmnote{#3}}

\theoremstyle{plain}

\newtheorem{theorem}[subsection]{Theorem}

\newtheorem{lemma}[subsection]{Lemma}
\newtheorem{corollary}[subsection]{Corollary}

\newtheorem{proposition}[subsection]{Proposition}

\theoremstyle{remark}

\theoremstyle{definition}
\newtheorem{conjecture}[subsection]{Conjecture}

\newtheorem{hypothesis}[subsection]{Hypothesis}
\numberwithin{equation}{section}

\theoremstyle{remark}
\newtheorem{remark}[subsection]{Remark}
\newtheorem*{claim}{Claim}

\theoremstyle{citing}

\makeatletter
\newsavebox\myboxA
\newsavebox\myboxB
\newlength\mylenA

\newcommand*\xtilde[2][0.8]{%
    \sbox{\myboxA}{$\m@th#2$}%
    \setbox\myboxB\null
    \ht\myboxB=\ht\myboxA%
    \dp\myboxB=\dp\myboxA%
    \wd\myboxB=#1\wd\myboxA
    \sbox\myboxB{$\m@th\widetilde{\copy\myboxB}$}
    \setlength\mylenA{\the\wd\myboxA}
    \addtolength\mylenA{-\the\wd\myboxB}%
    \ifdim\wd\myboxB<\wd\myboxA%
       \rlap{\hskip 0.5\mylenA\usebox\myboxB}{\usebox\myboxA}%
    \else
        \hskip -0.5\mylenA\rlap{\usebox\myboxA}{\hskip 0.5\mylenA\usebox\myboxB}%
    \fi}

\newbox\usefulbox

\def\getslant #1{\strip@pt\fontdimen1 #1}

\def\xxtilde #1{\mathchoice
 {{\setbox\usefulbox=\hbox{$\m@th\displaystyle #1$}%
    \dimen@ \getslant\the\textfont\symletters \ht\usefulbox
    \divide\dimen@ \tw@ 
    \kern\dimen@ 
    \xtilde{\kern-\dimen@ \box\usefulbox\kern\dimen@ }\kern-\dimen@ }}
 {{\setbox\usefulbox=\hbox{$\m@th\textstyle #1$}%
    \dimen@ \getslant\the\textfont\symletters \ht\usefulbox
    \divide\dimen@ \tw@ 
    \kern\dimen@ 
    \xtilde{\kern-\dimen@ \box\usefulbox\kern\dimen@ }\kern-\dimen@ }}
 {{\setbox\usefulbox=\hbox{$\m@th\scriptstyle #1$}%
    \dimen@ \getslant\the\scriptfont\symletters \ht\usefulbox
    \divide\dimen@ \tw@ 
    \kern\dimen@ 
    \xtilde{\kern-\dimen@ \box\usefulbox\kern\dimen@ }\kern-\dimen@ }}
 {{\setbox\usefulbox=\hbox{$\m@th\scriptscriptstyle #1$}%
    \dimen@ \getslant\the\scriptscriptfont\symletters \ht\usefulbox
    \divide\dimen@ \tw@ 
    \kern\dimen@ 
    \xtilde{\kern-\dimen@ \box\usefulbox\kern\dimen@ }\kern-\dimen@ }}%
 {}}

\newcommand*\xoverline[2][0.75]{%
    \sbox{\myboxA}{$\m@th#2$}%
    \setbox\myboxB\null
    \ht\myboxB=\ht\myboxA%
    \dp\myboxB=\dp\myboxA%
    \wd\myboxB=#1\wd\myboxA
    \sbox\myboxB{$\m@th\overline{\copy\myboxB}$}
    \setlength\mylenA{\the\wd\myboxA}
    \addtolength\mylenA{-\the\wd\myboxB}%
    \ifdim\wd\myboxB<\wd\myboxA%
       \rlap{\hskip 0.5\mylenA\usebox\myboxB}{\usebox\myboxA}%
    \else
        \hskip -0.5\mylenA\rlap{\usebox\myboxA}{\hskip 0.5\mylenA\usebox\myboxB}%
    \fi}

\def\xxoverline #1{\mathchoice
 {{\setbox\usefulbox=\hbox{$\m@th\displaystyle #1$}%
    \dimen@ \getslant\the\textfont\symletters \ht\usefulbox
    \divide\dimen@ \tw@ 
    \kern\dimen@ 
    \overline{\kern-\dimen@ \box\usefulbox\kern\dimen@ }\kern-\dimen@ }}
 {{\setbox\usefulbox=\hbox{$\m@th\textstyle #1$}%
    \dimen@ \getslant\the\textfont\symletters \ht\usefulbox
    \divide\dimen@ \tw@ 
    \kern\dimen@ 
    \xoverline{\kern-\dimen@ \box\usefulbox\kern\dimen@ }\kern-\dimen@ }}
 {{\setbox\usefulbox=\hbox{$\m@th\scriptstyle #1$}%
    \dimen@ \getslant\the\scriptfont\symletters \ht\usefulbox
    \divide\dimen@ \tw@ 
    \kern\dimen@ 
    \xoverline{\kern-\dimen@ \box\usefulbox\kern\dimen@ }\kern-\dimen@ }}
 {{\setbox\usefulbox=\hbox{$\m@th\scriptscriptstyle #1$}%
    \dimen@ \getslant\the\scriptscriptfont\symletters \ht\usefulbox
    \divide\dimen@ \tw@ 
    \kern\dimen@ 
    \xoverline{\kern-\dimen@ \box\usefulbox\kern\dimen@ }\kern-\dimen@ }}%
 {}}
\makeatother

\theoremstyle{definition}

\title[Zariski density of rational curves]{On the Zariski density of rational curves on IHS manifolds}

\author{Pietro Beri}
\address{Pietro Beri\\ 
Universit\'e de Lorraine, CNRS, IECL\\
F-54000 Nancy -- France }
\email{pietro.beri@univ-lorraine.fr}

\author{Giovanni Mongardi}
\address{Giovanni Mongardi\\Alma Mater Studiorum, Università di Bologna,  P.zza di porta san Donato, 5, 40126 Bologna, Italia}
\email{giovanni.mongardi2@unibo.it}

\author{Gianluca Pacienza}
\address{Gianluca Pacienza\\ 
Universit\'e de Lorraine, CNRS, IECL\\
F-54000 Nancy -- France }
\email{gianluca.pacienza@univ-lorraine.fr}

\begin{document}


\begin{abstract}
   In analogy with recent works on $K3$ surfaces, we study the existence of infinitely many ruled divisors on projective irreducible holomorphic symplectic (IHS) manifolds. We prove such an existence result for any projective  IHS manifold of $K3^{[n]}$ or generalized Kummer type which is not a variety defined over $\overline{\Q}$ with Picard number one or maximal. The result is obtained as a combination of the regeneration principle and of a generalization to higher dimension of a controlled degeneration technique, invented by Chen, Gounelas and Liedtke in dimension 2.  
\end{abstract} 
\maketitle
\section{Introduction}
The following conjecture has attracted considerable attention in recent years. 

\begin{conjecture}[cf. \cite{Huy16}, Chapter 13, Conjecture 0.2]\label{conj:K3}
    Every projective K3 surface $(X,H)$ over an algebraic closed field contains infinitely many integral rational curves $C$ linearly equivalent to some multiple of $H$.
\end{conjecture}
We refer the reader to \cite[Chapter 13, Section 0.3]{Huy16} for a discussion of the context, provided by (non-)hyperbolicity type questions and conjectures on rational points on $K3$ surfaces, in which this conjecture has emerged.  Conjecture \ref{conj:K3} has been proven when the characteristic of the field is not 2 or 3 and the K3 surface is very general by Chen, Gounelas and Liedtke in \cite[Theorem A]{CGL} as a combination of two new techniques that they introduced for the study of rational curves on $K3$ surfaces. We refer the reader to  \cite{BHT, CGL, CGL22b, Huy16} for a  list of previous significant contributions in the direction of the conjecture.  

In view of the strong analogies between K3 surfaces and IHS manifolds, it is natural and tempting to make the following analogous prediction (we limit ourselves to the field of complex numbers):
\begin{conjecture}\label{conj:strong}
    Every projective IHS manifold $(X,H)$ contains infinitely many integral rational curves $C$, each of which rules a divisor linearly equivalent to some multiple of $H$.
\end{conjecture}
To our knowledge, the above conjecture is not present in the literature, although expected by experts in the field. 

Partial evidence for this conjecture is provided by \cite{Reg}, where it is proven for \emph{very general} projective IHS manifolds of $K3^{[n]}$ and generalized Kummer type. However, the very general set where the conjecture is proven in loc. cit. is not explicit, therefore it is difficult to apply the result in practice for a given IHS manifold. 
In this paper, we 
concentrate on the following slightly weaker statement,   with the hope of obtaining more precise results.
\begin{conjecture}\label{conj:weak}
    Every projective IHS manifold $X$ contains infinitely many integral rational curves $C$, whose degrees with respect to a polarization are unbounded and each of which rules an ample divisor.
\end{conjecture}
The weaker conjecture has been proven in \cite{Reg} for all IHS manifolds $X$ of $K3^{[n]}$ or generalized Kummer type such that $\Bir(X)$ is infinite, a condition that can be easier to check in practice.
The main result of this paper is the following:
\begin{theorem}\label{thm:main}
    Let $\mathcal M$ be any irreducible component of a moduli space of projective IHS manifolds of $K3^{[n]}$ or generalized Kummer type. 
     Then the following holds:
     \begin{itemize}
\item[(1)] All manifolds in $\mathcal{M}$ with Picard number at least two and strictly smaller than the maximal rank contain infinitely many distinct ample ruled divisors. 
\item[(2)] If $[X]\in \mathcal M$ corresponds an IHS manifold which is not defined over $\overline{\Q}$, then X contains infinitely many distinct ample ruled divisors.
\end{itemize}

\end{theorem}
\begin{remark}\label{rem_finite_bir}
    In a forthcoming paper \cite{Kir25}, Kirschmer proves that there are no reflective lorentzian lattices of rank 21. These are lattices of signature $(1,20)$ whose isometry group is a finite extension of the group generated by reflections. By the proof of the birational Morrison--Kawamata cone conjecture for IHS manifolds \cite[Theorem 6.25]{Mar11}, it follows that a IHS manifold with finite birational automorphism group has Picard lattice which is reflective lorentzian. Therefore, manifolds of $K3^{[n]}$ type with maximal Picard rank have always infinite birational automorphism groups and they have infinitely many distinct big and nef ruled divisors by \cite[Theorem 1.7]{Reg}. 
\end{remark}
Notice however that, as a consequence of \cite[Theorem 1.1]{Maulik12}, we know that there exist projective IHS manifolds defined over $\overline{\Q}$ with Picard number one, cf. Proposition \ref{prop:Qbarra_rho1}. Even more is true in dimension two, as in \cite[Theorem 1.1]{van07} the author proves the Zariski density in moduli of quartics defined over $\Q$ with Picard number one (and infinitely many rational points). 
Yet as an immediate consequence of Theorem \ref{thm:main} we have the following:
\begin{corollary}
\label{thm:mainbis}
    Let $\mathcal M$ be any irreducible component of a moduli space of projective IHS manifolds of $K3^{[n]}$ or generalized Kummer type. 
     Then for any algebraic curve $B\subset \mathcal{M}$, the set of points in $B$ not containing infinitely many distinct ample ruled divisors is at most countable. 
\end{corollary}

Actually, Conjecture \ref{conj:weak} for all varieties over $\overline{\Q}$ should be considered the hardest case, as, using \cite[Regeneration principle 1.3]{Reg}  we can show the following (see Theorem \ref{thm:weak_to_strong_technical} for a more general statement):
\begin{theorem}\label{thm:weak_to_strong}
    Suppose that Conjecture \ref{conj:weak} holds for all projective IHS varieties of $K3^{[n]}$ or of generalized Kummer type defined over $\overline{\Q}$. Then Conjecture \ref{conj:weak} holds for all manifolds in the same deformation class.
\end{theorem}

Theorem \ref{thm:main} will be proven using a controlled degeneration principle, which allows to control the integrality of degenerations of rational curves and may be interesting in its own. We will make use of moduli spaces of lattice-(pseudo)-polarized IHS manifolds for which we refer the reader to \cite{Dol96, BCS_class}.

\begin{theorem}\label{thm:degprin}
(Controlled degeneration principle) Let $X$ be a projective IHS manifold with $\Pic(X)=\Lambda$ for a lattice $\Lambda$ of rank at least two. Let $L\in \Lambda$ be a primitive ample divisor on $X$. Let $\mathcal{M}$ be an irreducible component of the moduli space $\mathcal{M}_\Lambda$ of $\Lambda$-polarized IHS containing $(X,\Lambda)$. Suppose that for a very general point $X_t\in \mathcal{M}$ there is $m\in \mathbb{N}$ and a ruled divisor in $|mL_t|$ ruled by an integral primitive curve. Then the same holds for $X$.
\end{theorem}

This principle is a generalization to IHS of arbitrary dimension of \cite[Theorem C]{CGL} and it is proven along the same lines.

Notice that the results contained in this paper allow to obtain existence of infinitely many ruled divisors for other classes of irreducible symplectic varieties based upon the following simple observation: if $X \dashrightarrow Y$ is a dominant rational map and $X$ contains a infinitely many ruled divisors, then $Y$ does. In this way we deduce the  existence of infinitely many ruled divisors with positive square with respect to the Beauville-Bogomolov-Fujiki form for some $OG6$ (exactly for those dominated by a $K3^{[3]}$, see \cite{MRS}), as well as for most of the singular symplectic varieties studied in \cite{BGMM}. We refer the reader to Section \ref{sec:dominati} for more details. 

For the basic theory of IHS manifolds we refer the reader to the standard references \cite{Bea83, Huy99}.

{\bf Acknowledgements.} We thank  F. Charles for bringing \cite{Maulik12} to our attention, S. Floccari for pointing out Proposition \ref{prop:dominati}, S. Brandhorst, M. Kirschmer, S. Tayou and K. O'Grady for valuable comments.
\bigskip

{\small G.M. was supported by PRIN2020 research grant ``2020KKWT53'', by PRIN2022 research grant ``2022PEKYBJ'' and is a member
of the INDAM-GNSAGA. P.B. and G.P. were  supported by the project ANR-23-CE40-0026 ``Positivity on K-trivial varieties''.
This work was also supported by the "National Group for Algebraic
and Geometric Structures, and their Applications" (GNSAGA - INDAM), in particular part of the work was carried out while G.P. was visiting professor under the support of GNSAGA.
Funded by the European Union - NextGenerationEU under the National Recovery and Resilience Plan (PNRR) - Mission 4 Education and research - Component 2 From research to business - Investment 1.1 Notice Prin 2022 - DD N. 104 del 2/2/2022, ``Symplectic varieties: their interplay with Fano manifolds and derived categories'', proposal code 2022PEKYBJ – CUP J53D23003840006.}

\section{Controlled degenerations}\label{s:deg}
In this section we will give a proof of Theorem \ref{thm:degprin}.
The main ingredient of the proof  will be the following result.
\begin{proposition}[\cite{CGL}, Proposition 7.4]\label{prop:CGL}
Let
\begin{equation}
    \xymatrix{\sC \ar[d]_{f}\ar[r] & V\ar[d]^{\varphi} & E\ar@{_{(}->}[l]\ar[d]\\
    \sS \ar[r]^{\pi}& U & \{0\}\ar@{_{(}->}[l]}
\end{equation}
be a commutative diagram of quasi-projective varieties over an algebraically closed field, where
\begin{itemize}
\item[-] all morphisms are projective;
\item[-]  $\varphi:V\to U$ is a dominant map between two smooth quasi-projective surfaces $V$ and $U$;
\item[-] $o\in U$ is a point of $U$, $E$ is a connected component of $\varphi^{-1}(0)$ and
\item[-] $f:\sC/V\to \sS$ is a family of stable maps of genus $0$ to $\sS$ over $V$ whose very general fibres $\sC_b$ are smooth for $b\in V$.
\end{itemize}
Suppose there is a closed subscheme $Y$ of $\sC_E:=\sC\times_V E$ satisfying that $Y$ is flat over $E$ of pure dimension $\dim E+1$,
\begin{equation}\label{eq:rig}
    f(Y_s)=f(Y) \textrm{ for all } s\in E,
\end{equation}
and there are a point $e\in E$ and two distinct irreducible components $G_1$ and $G_2$ of $Y_e$ with $f_*G_i\not= 0$ for $i=1,2$. Then there exists an integral curve $F\subset V$ such that $\varphi_* F\not= 0$,  $E\cap F\not=\emptyset$ and $\sC_p$ is singular for all $p\in F$.
\end{proposition}

\begin{proof}[Proof of the Controlled degeneration principle \ref{thm:degprin}]
The main point of the proof is  to  reduce to a sort of relative version of the above proposition \cite[Proposition 7.4]{CGL} by applying it to a very general fibre of a family of curves ruling a divisor. This is possible by the choice of a section, after possibly a base change. We make things precise in what follows.

 As usual $2n$ will denote the dimension of $X$.
We first deal with the case $\rho(X)=2$, where the idea of the proof is more transparent.
We consider one component of the moduli space $\mathbf M$ of polarized IHS manifolds containing the point $[(X,L)]$. In this moduli space $\mathbf M$ we may take  a general smooth and irreducible affine surface $U$  passing through the point $(X,L)$. Since $\sM_{\Lambda}$ is smooth,  
the surface $U$ can be taken so that there is precisely one irreducible curve $B\subset U$ containing  $[(X,L)]$ and parametrizing IHS manifolds $\sX_b,\ b\in B$ with $\Lambda\subset \Pic(\sX_b)$. Hence, if $(\sX,\sL)\to U$ is the corresponding family of IHS manifolds, we have $\Pic(\sX_u)=\mathbb Z \sL_u$ for a very general $u\in U$,  $\Pic(\sX_b)=\Lambda$ for a very general $b\in B$, and $\sX_0=X$ and $\sL_0=L$ for $0\in B$.

Furthermore by hypothesis 
 we can choose $U$ such that for a  very general $b\in B$ and  some $m>0$, the manifold $\sX_b$ contains a ruled divisor $D_b\in |m\sL_b|$ ruled by integral curves with {\it primitive} class.
 Let 
 \begin{equation}\label{eq:alpha}
 \alpha\in H^0(U,R^{4n-2}f_* \Z)
 \end{equation}
 be a global section of Hodge type $(2n-1,2n-1)$ of $R^{4n-2}f_* \Z$ such that $\alpha_b$ is the class of primitive curves ruling $D_b$.  By \cite[Proposition 3.1]{CMP19} these curves deform over $U$ to yield, for all $u\in U$ a ruled divisor $D_u\in |m\sL_u|$.
 The locus in $D_u$ given by points where a positive dimensional family of rational curves passes is a closed subset of smaller dimension for all $u\in U$ by \cite[Theorem 1.3 item (ii)]{V16}. Therefore, up to a base change we may and will assume that for any point of $D_u$ outside of the closed locus above, there exists a unique rational curve $R_u\subset \sX_u$ in the family ruling $D_u$. Furthermore, for $u\in U$ very general, this curve is irreducible. Notice that by hypothesis this holds also on very general points of $B$.

 We then consider the family $\sR\subset \sX$ of such rational curves, which by construction contains the central point $0$. 
  By eventually applying stable reduction, cf. \cite[Proposition 8.2]{CGL}, we then obtain the following diagram
\begin{equation}\label{eq:diag}
    \xymatrix{\sC \ar[d]_{f}\ar[r] & V\ar[d]^{\varphi} & E\ar@{_{(}->}[l]\ar[d]\\
    \ \ \ \ \ \ \sR \subset \sX \ar[r]^{\ \ \ \ \ \pi}& U & \{0\}\ar@{_{(}->}[l]}
\end{equation}
where
\begin{itemize}
\item[(i)] all morphisms are projective;
\item[(ii)] $V$ is a smooth quasi-projective surface, mapped surjectively and generically finitely onto $U$ by $\varphi$;
\item[(iii)] $E$ is a connected component of $\varphi^{-1}(0)$;
\item[(iv)] $f:\sC/V\to \sR\subset \sX$ is a flat family of stable maps of genus $0$ to $\sR\subset \sX$ over $V$;
\item[(v)] $\sC_v$ is smooth for very general $\varphi(v)\in B$.
\end{itemize}

Notice that we must have
\begin{equation}\label{eq:rigid}
   f(\sC_e)=f(\sC_E)
\end{equation}
for all $e\in E$, where $\sC_E:=\sC\times_V E$.
Indeed if (\ref{eq:rigid}) did not hold, fix an irreducible component of $f(\sC_E)$: we would have a rationally chain connected surface passing through the very general point of this irreducible ruled divisor inside the IHS manifold $\mathcal{X}_0$, thus contradicting \cite[Theorem 1.3, item (ii)]{V16}. Notice that (\ref{eq:rigid}) is crucial as it corresponds to hypothesis (\ref{eq:rig}) in Proposition \ref{prop:CGL}. It is automatic in the $2$-dimensional case, as a (complex) $K3$ surface cannot be covered by rational curves, while it requires the finer analysis done in \cite{V16} in higher dimensions.

\begin{claim}
The curve $\sC_e$ is smooth for all $e\in E$.
\end{claim}
By stability, if this were not the case we could write $\sC_e=G_1+G_2+G_{more}$,  where $G_1,G_2$  are two irreducible components of $\sC_e$ such that $f_*(G_i)\not=0,\ i=1,2$. We could then apply Proposition \ref{prop:CGL} with $Y=\sC_E$ and get an integral curve $F\subset V$ such that $\varphi_* F\not= 0$, $E\cap F\not=\emptyset$ and $\sC_p$ is singular for all $p\in F$.  Therefore let us write $\sC_p=M_p+ N_p$ for all $p\in F$ where $M_p$ and $N_p$ are union of irreducible components of $\sC_p$ with $f_*M_p\not=0\not=f_*N_p$. In particular, for $p\in E\cap F$ we would have $[f_*M_p]+ [f_*N_p]=\alpha_0$. 
Here we denoted by $[f_*M_p]$ and $[f_*N_p]$ the cohomology classes of $1$-cycles in $\mathcal{X}_p$ and $\alpha_0$ is as in (\ref{eq:alpha}).
As $\alpha_0$ is primitive, the two components $[f_*M_p],\, [f_*N_p]$ generate a rank two $\Z$-module. By hypothesis, the dual of $\Lambda$ is given by algebraic classes of type $(2n-1,2n-1)$ on $X$ which therefore form a $\Z$-module of rank two.

  This implies that the curve $F$ obtained by Proposition \ref{prop:CGL} coincides with the curve $B$ giving locally the Hodge locus of $\Lambda$ inside $U$. However, by hypothesis, the very general element of $B$ corresponds to non-singular curves, which contradicts item (v) of diagram (\ref{eq:diag}) and the Claim is proved.

As a consequence of the Claim above $f(\sC_e)$ is an integral curve for all $e\in E$, hence in particular for $X$ and the theorem is proved in this case.

We now deal with the general case $\rho(X)\geq 2$.
As before we consider an irreducible component $\sM_{\Lambda}$ of the lattice polarized moduli space $\sM_{\Lambda}$. For all lattices $\Sigma\subset \Lambda$ of corank one such that $L\in\Sigma$, we can pick a surface $U_\Sigma$ in the moduli space of $\Sigma$-polarized IHS manifolds and an irreducible curve $B_\Sigma$ in it such that $\Lambda\subset \Pic(\sX_b)$ for all $b\in B_\Sigma$ and $\Sigma\subset \Pic(\sX_u)$ for all $u\in U_\Sigma$. Moreover, we obtain a diagram identical to (\ref{eq:diag}) where all items (i)--(v) hold, as well as (\ref{eq:rigid}). We denote again with $\alpha$ a section of the local system $R^{4n-2}f_*\Z$ as in (\ref{eq:alpha}). 
We now make a careful choice of $\Sigma$ such that also the last part of the proof goes through.
We denote with $\Lambda^\vee$ the dual lattice of $\Lambda$ contained in $H_2(X,\Z)$.
We claim that there exists a primitive sublattice $\Sigma\subset \Lambda$ with $L\in\Sigma$, whose dual we denote with $\Sigma^\vee$ such that
\begin{enumerate}
\item[(j)] $\rk (\Sigma)=r-1< r=\rk(\Lambda)$.  
\item[(jj)] $\Sigma^\vee\cap \{ \sum_{i=1}^k a_i [\Gamma_i]:\! 0\leq a_i\leq n_i,\ a_i\in \mathbb Z,\ 0<\sum_{i=1}^k a_i<\sum_{i=1}^k n_i\}= \emptyset .$
\end{enumerate}
Here, $\alpha_0=\sum_1^n \mu_i [\Gamma_i]$ with $\mu_i\in\Z^+$ and $\Gamma_i$ are irreducible and primitive components of $\alpha_0$, which represent a flat limit of the curves ruling the divisors $D_b$ for $b\in B_\Sigma$ very general. 
The proof of the above claim is exactly as in the case of $K3$ surfaces, the fact that we are considering a lattice with a $\Q$-valued form instead of $\Z$-valued does not play a role, cf. \cite[Proof of Theorem 7.1]{CGL}. 

Then the rest of the proof goes as in the $\rho=2$ case: for all $p\in F$ we write $\sC_p=M_p+N_p$ such that $f_*M_p\neq 0\neq f_*N_p$ and by construction we have
$$
 0<[f_* M_p],\ [f_*N_p] <\alpha_0.  
$$
Hence, by item (jj) above $[f_* M_p]$ and $[f_*N_p] $ do not belong to $\Sigma^\vee$. 

By item (j) we have $\Pic(\sX_{\varphi(p)})\supset \Lambda,\ \forall p\in F$. Therefore $B=\varphi(F)$ which yields a contradiction to item (v), since $\sC_p$ is smooth for $\varphi(p)\in B$ very general.
\end{proof}

\section{Applications}
The aim of this section is to prove the two items of Theorem \ref{thm:main}. 
The proof of Theorem \ref{thm:main}, item (1),  will make use of the following lemma:
\begin{lemma}\label{lem:densità}
    Let $\mathcal{M}$ be a component of the moduli space of $\Lambda$ lattice-polarized IHS manifold of $K3^{[n]}$ or generalized Kummer type, where $\Lambda$ is an even indefinite lattice of rank between 2 and $b_2-3$. Then the locus of points of $\mathcal{M}$ which are birational to a Hilbert scheme of points on a K3 surface or to a generalized Kummer manifold is dense.  
\end{lemma}
\begin{proof}

    By \cite[Corollary 3.18]{MPdens}, we only need to prove that the locus in the statement is non-empty. Notice that the statement of  \cite[Corollary 3.18]{MPdens} itself is about the density of moduli spaces of sheaves (resp. their Albanese fibres), however its first proof proves more precisely the density of points birational to Hilbert schemes (resp. generalized Kummers). 

    We prove non-emptyness for the $K3^{[n]}$ type, the proof for the generalized Kummer type is identical with the obvious modifications.
    We denote $L=U^{\oplus 3}\oplus E_8(-1)^{\oplus 2}\oplus \mathbb{Z}\ell$ with $\ell^2=-2(n-1)$. Observe that any sublattice $R\subset L$ containing $\ell$ splits as $\mathbb{Z}\ell\oplus \ell^{\perp_R}$, by \cite[Lemma 7.2]{GHS11} and \cite[Chapter 14, (0.2)]{Huy16}.
    A component $\mathcal{M}$ corresponds to an orbit of isometric embeddings
    $\Lambda\hookrightarrow L$; we fix one such embedding. If $\Lambda\subset L$ contains $\sigma(\ell)$ for some $\sigma$ in the $M$-polarized monodromy, by the observation above a very general point $p$ in the period domain of $\mathcal{M}$ corresponds to the same Hodge structure of the Hilbert scheme of points on some K3 surface $S$, whose Néron-Severi is isomorphic to $\ell^{\perp_\Lambda}$. By the surjectivity of the period map for marked IHS manifolds and the Global Torelli Theorem, there is a $\Lambda$-polarized IHS, birational to $S^{[n]}$, whose period point is $p$. By \cite[Theorem 4.6]{Huy99}, we know that birational $\Lambda$-polarized IHS correspond to inseparable points in the moduli space. Hence the non-emptyness of IHS birational to $S^{[n]}$ yields the non-emptyness of IHS isomorphic to an Hilbert scheme once we prove that 
    the nef cone of $S^{[n]}$ has non-empty intersection with the copy of $\Lambda$ in $\Pic(S^{[n]})$: this is clear, since the two lattices are the same for the very general case.

    If instead $\sigma(\ell)\notin \Lambda$ for any $\sigma$, we consider the saturated $R$ of $\langle \Lambda,\ell\rangle$ and we repeat the argument. Note that $R$ is in the form $N\oplus \mathbb{Z}\ell$, by the same lattice-theoretic argument as above. In this case the Néron-Severi group of $S$ is isometric to $N$;
    its intersection with the nef cone of $S^{[n]}$ lies then in the wall associated to the Hilbert-Chow contraction, hence any non-zero element in the intersection $\Lambda\cap N$, whose rank is at least one, yields a pseudo-polarization of the Hilbert scheme and we are done.
%
\end{proof}

\begin{proof}[Proof of Theorem \ref{thm:main}, item (1)]
Let $\Lambda$ be a lattice of rank between two and $b_2(X)-3$ such that $\Lambda= \Pic(X)$ and there is an ample class contained in $\Lambda$. Let us consider the moduli space of $\Lambda$-polarized IHS manifold, and let $\mathcal{M}$ be a component that contains $X$ with the above given inclusion $\Lambda= \Pic(X)$. 
Thanks to Lemma \ref{lem:densità}, we can pick a point in the component of the moduli space in which $X$ lies and which parametrizes the
punctual Hilbert scheme $S^{[n]}$ of a projective $K3$ surface $S$ (resp. the generalized Kummer manifold $K_n(A)$ of an abelian variety).
We therefore have a family $f:\mathcal X\to B$ of projective IHS manifolds such that for a point $b_0\in B$ we have $X_0:=f^{-1}(b_0)\cong X$, for $b_1\in B$ we have $X_1:=f^{-1}(b_1)\cong S^{[n]}$ (resp. $\cong K_n(A)$) and $\forall b\in B$ the Picard number of $X_b:=f^{-1}(b)$ is at least 2. We know that $S^{[n]}$ contains infinitely many divisors ruled by integral rational curves (by taking infinitely many rational curves in $S$) and ditto for $K_n(A)$ (this time by the proof of \cite[Theorem 1.6, item (2)]{Reg}). By  \cite[Theorem 1.4]{Reg} we can regenerate these divisors and get infinitely many ample ruled divisors for very general $b\in B$. Then it suffices to apply the controlled degeneration principle Theorem \ref{thm:degprin} to conclude. 
\end{proof}

Having Theorem \ref{thm:main}, \textit{item (1)} at hand, the proof of the second item follows a similar statement for $K3$ surfaces which can be found in \cite[Chapter 13, Theorem 3.1]{Huy16}. 
We start with the following. 

\begin{lemma}\label{lem:basechange}
    Let $X$ be a projective variety defined over a field $k$ and $k\subset K$ a field extension. Consider $X_K:=X\times_k K\to X$ and the pull-back map 
    $$\Pic(X)\to \Pic(X_K),\ L\mapsto L_K.$$ Then
    \begin{itemize}
        \item[1)] The pull-back map is  injective.  
        \item[2)] If $k$ is algebraically closed and $h^1(\mathcal{O}_X)=0$ the pull-back map is also surjective. 
    \end{itemize}
\end{lemma}
\begin{proof}
    1) By flat base-change, for any $L\in \Pic(X)$ we have $H^i(X_K,L_K)=H^i(X,L)\otimes_k K$. We have that $L$ (respectively $L_K$) is trivial if and only  $h^0(X,L)\neq 0 \neq h^0(X,-L)$ (resp. $h^0(X_K,L_K)\neq 0 \neq h^0(X_K,-L_K)$) and the statement follows.\\ 
    2) The proof goes exactly as in \cite[Chapter 17, Lemma 2.2]{Huy16}.
\end{proof}
We have moreover the following:
\begin{lemma}\label{lem:finite_bad_points}
    Let $B$ a positive dimensional subscheme of a component $\mathcal{M}$ of the moduli space of polarized IHS manifolds with $b_2\geq 5$. Let $B_{fin}$ be the subset of $B$ corresponding to points $(X,H)$ such that $X$ has maximal Picard rank and $Bir(X)$ is finite. Then the set $B_{fin}$ is a finite set.
\end{lemma}
\begin{proof}
    By the birational Morrison--Kawamata cone conjecture \cite[Theorem 6.25]{Mar11}, a IHS manifold with finite birational automorphism group has a reflective lorentzian Picard lattice. By \cite[Theorem 5.2.1]{Nik82}, in any given rank (at least three), these lattices fall into finitely many isometry classes, up to scaling. As we are restricting to IHS manifolds with $b_2\geq 5$, this applies to all lattices $\Pic(X)$ such that $(X,H)\in B_{fin}$ for some $H$. Furthermore, we have that $Bir(X)$ is commensurable with $O(\Pic(X))/W_{Pex}(\Pic(X))$, where $W_{Pex}(\Pic(X))$ is the group generated by reflections on stably prime exceptional divisors by \cite[Theorem 6.18, item (5)]{Mar11}. As stably prime exceptional divisors have bounded square in any deformation class as a consequence of \cite[Proposition 6.2]{Mar11}, for any reflective lorentzian lattice there is a finite number (possibly zero) of rescalings such that it corresponds to the Picard lattice of some $(X,H)\in B_{fin}$. This implies that the natural forgetful projection $B_{fin}\to \{\text{IHS manifolds}\}/\sim$ mapping $(X,H)$ to the isomorphism class of $X$ has finite image. However, by \cite[Theorem 0.1]{Huy18}, the fibers of this forgetful projection are finite, hence the lemma follows.
\end{proof}
Notice that, by \ref{rem_finite_bir}, in the case of manifolds of $K3^{[n]}$ type the set $B_{fin}$ is actually empty.
\begin{proof}[Proof of Theorem \ref{thm:main} item (2)]
Let $(X,H)$ be a polarized IHS manifold of $K3^{[n]}$ or Kummer type, which cannot be defined over $\overline{\Q}$. Let $\mathcal{M}$ be a component of the moduli space of polarized IHS manifolds containing $(X,H)$. We consider $\sM$ as the moduli space defined over $\overline{\Q}$, and the complex moduli space as its base change $\sM \times_{\overline{\Q}} \C$. As $X$  cannot be defined over $\overline{\Q}$, the point $(X,H)\in \mathcal{M}$ is not closed. We denote by $B$ its closure. 
Let $\mathcal{X}\to B$ be the associated family of polarized IHS manifolds over $B$.
By construction we have that $X\times_\C k(\eta)\cong X_\eta$, where $\eta\in B$ is the generic point. 

\begin{claim}
    There is a bijection between the ruled divisors in $X$ and those in $X_\eta$.
\end{claim}
Indeed, as in \cite[Chapter 17, Lemma 2.2]{Huy16}, a line bundle $M$ on $X_\eta$ can be seen as a family of line bundles over $X$, parametrized by a base $\textrm {Spec}(A)$ (where $A$ is an algebra whose quotient field is $k(\eta))$. A section of $M$ giving a ruled divisor on $X_\eta$ gives therefore a family of sections over $\textrm {Spec(A)}$ giving a ruled divisor on each fibre $X_a\cong X$. However, as ruled divisors on an IHS manifold are rigid in their linear systems, these sections need to be locally constant, hence they correspond to a single ruled divisor on $X$.

\medskip

By Lemma \ref{lem:basechange}, the Picard number $\rho(X_\eta)$ of $X_\eta$ equals the Picard number of $X$. Therefore the locus
$$B_0=\{b\in B\text{ such that } \rho(\mathcal{X}_b) >\rho(X_\eta)\}$$
is a dense subset of $B$, as the family over $B$ is not isotrivial.

Therefore, by Theorem \ref{thm:main}, item (1), we know that for all points $b\in B_0$ that do not have maximal Picard rank there are countably many distinct ruled divisors, and we take a regeneration of them using \cite[Theorem 1.4]{Reg}. As we are regenerating divisors which are distinct, their regenerations are also distinct over the generic point $X_\eta$.
By the above Claim we obtain countably many distinct ruled divisors on $X$.\\
It remains to prove the theorem in the case when all points of $B_0$ have maximal Picard rank. If such a $(X_b,H_b)$, $b\in B_0$, has infinite birational automorphism group, we can still use the same argument as before, as by \cite{Reg} $X_b$ will have countably many distinct uniruled divisors. However, by Lemma \ref{lem:finite_bad_points}, the subset $B_{fin}$ of $B_0$ consisting of pairs $(X_b,H_b)$ with maximal Picard rank and finite birational automorphism group is finite, hence there are at least countably many points in $B_0\setminus B_{fin}$ and the Theorem follows.
\end{proof}

We record the following improvement of \cite[Theorem 1.6]{Reg} which will be needed in the proof of Theorem \ref{thm:mainbis}. 

\begin{theorem}\label{thm:rmk}
    Let $f:\mathcal{X}\to B$ be a family of projective IHS manifolds of $K3^{[n]}$ or the generalized
Kummer type, with a
 fibre $X_0$ such that $\rho(X_0)\geq 2$.  Then the very general fiber of $f$ contains infinitely many distinct ruled divisors.
\end{theorem}
\begin{proof}[Proof of Theorem \ref{thm:rmk}]
By Theorem \ref{thm:main}, item (1), the fibre $X_0$ contains infinitely many distinct ample ruled divisors. 
By \cite[Theorem 1.4]{Reg} we can regenerate all of them.    
These regenerations remain distinct and irreducible when restricted to the very general fiber of the family and we are done.
\end{proof}

\begin{proof}[Proof of Corollary \ref{thm:mainbis}]
    Let $B$ be a curve inside $\mathcal{M}$. We claim that there is a point $0\in B$ with $\rho(X_{0})\geq 2$, actually more is true, the locus of these points is dense in $B$. To see the latter, one can proceed as in the classical case of K3 surfaces, see \cite[Chapter 6, Proposition 2.9]{Huy16}, as the proof only uses the local Torelli theorem and the structure of the period domain, which are analogous in the IHS case. 
    The result now follows from the combination of the first item of Theorem \ref{thm:main} and Theorem \ref{thm:rmk}. Alternatively one can also notice that, by Theorem \ref{thm:main}, item (i), the points of $B$ where we cannot conclude the existence of infinitely many amply ruled divisors are those defined over $\overline{\Q}$ and these are at most countable. 
\end{proof}

\section{From $\overline{\Q}$ to $\mathbb C$ and from weak to strong}

In this section we investigate some consequences of the regeneration principle for IHS manifolds defined over a field of characteristic zero. First of all we have the following:
\begin{proposition}\label{prop:Qbarra_rho1}
Let $\mathbb{K}$ be an algebraically closed field of characteristic zero. Then there are, in every deformation class, IHS manifold defined over $\mathbb{K}$ of Picard rank one.
\end{proposition}
\begin{proof}
    The proof is based on \cite[Theorem 1.1]{Maulik12}, which states that if we can find a family defined over $\mathbb{K}$, then there are closed points whose Picard rank equals that of the geometric generic fibre. To obtain our claim, it suffices to consider any component $\mathcal{F}$ of the moduli space of polarized IHS manifolds. As these moduli spaces are defined over $\overline{\Q}$ (actually over any field of characteristic zero by \cite[Section 2.3.4]{Andr96}, therefore we can consider the moduli space $\mathcal{F} \times_{\overline{\Q}} \mathbb{K}$ and any family over a subset of $\mathcal{F}$ whose geometric generic fibre has Picard rank one. Therefore, our claim follows.
\end{proof}
We will need to produce regenerations of divisors, so let us recall the condition under which the regeneration principle was proven in \cite{Reg}:

\begin{hypothesis}\label{hyp:curve che rigano}
    Let $X$ be a projective IHS manifold. There exists a constant $d\geq 0$ such that  all primitive ample curve classes $[C]\in H_2(X,\Z)$ satisfying $q(C)> d$   have a connected and
rational representative $R\in [C]$ such that $R$ rules a prime divisor of class proportional to $[C]^\vee$.
\end{hypothesis}

The final aim of this section is to prove Theorem \ref{thm:weak_to_strong}, which will follow from the following more general yet more technical statement:

\begin{theorem}\label{thm:weak_to_strong_technical}
    Let $\mathcal{M}$ be a component of the moduli space of polarized IHS manifolds and let $(\mathcal{X},\mathcal{H})\to \mathcal{M}$ be the universal family over it. Suppose moreover that Conjecture \ref{conj:weak} (respectively Conjecture \ref{conj:strong}) and Hypothesis \ref{hyp:curve che rigano} hold for all $\overline{\Q}$ points of $\mathcal{M}$. Then Conjecture \ref{conj:weak} (respectively \ref{conj:strong} for the polarization $\mathcal{H}_t)$ holds for all points $t$ of $\mathcal{M}$.
\end{theorem}
\begin{proof}

  Let $t\in\mathcal{M}$. If $t$ is defined over $\overline{\Q}$, there is nothing to prove. So let us suppose it is not defined over $\overline{\Q}$. The corresponding polarized $(X_t,H_t)$ can be defined over a finite extension of $\overline{\Q}$, which implies that we can assume that $(X_t,H_t)$ is defined over the function field of an affine variety $B=\textrm{Spec}(A)$ over $\overline{\Q}$ ($B$ can of course be a point). We can therefore assume, eventually after shrinking the base, that we have a family $(\mathcal{X},\mathcal{H})\to B$ whose geometric generic fibre is $(X_t,H_t)$.
Let us pick a point $b\in B(\overline{\mathbb{Q}})$. By hypothesis, we have countably many distinct ruled divisor $\{D_i\}_{i\in\N}$ inside ${X}_b$. Without loss of generality we can suppose that only finitely many of them are not proportional to $H_b$, and we can remove these finitely many from the above set. Let us call $C_i$ the class of the rational curve ruling $D_i$, and let ${h}$ be the relative class of the dual curve to $\mathcal{H}$. For any $i$, there is a $n_i$ such that $n_i{h}_b-C_i$ is ample, primitive and has square satisfying hypothesis \ref{hyp:curve che rigano}. 
Therefore, we can apply \cite[Regeneration Principle 1.3]{Reg} to obtain a regeneration of $D_i$ over $B$, which by construction is ruled by a curve of class $n_i{h}$. Notice that the $D_i$ are all distinct and form a countable set, hence the regenerations we obtain are still countably many. As $(X_t,H_t)$ is the generic point of this family, the restrictions of these regenerations to $(X_t,H_t)$ are still countably many and Conjecture \ref{conj:strong} is proven for it. 
\end{proof}
Notice that what we actually prove is that Conjecture \ref{conj:weak} for $\overline{\Q}$ points implies Conjecture \ref{conj:strong} for all other points.

\begin{proof}[Proof of Theorem \ref{thm:weak_to_strong}]
It suffices to use \cite[Proposition 2.1]{Reg}, which proves Hypothesis \ref{hyp:curve che rigano} for all manifolds of $K3^{[n]}$ or Kummer type. 
\end{proof}
\section{Existence from domination} \label{sec:dominati}
In this section we show how to use the main results of the paper to deduce similar results for other classes of varieties. We hope that this can be useful in other situations. The key observation is  the following.
    \begin{proposition}\label{prop:dominati}
    Let $f:Y\dashrightarrow X$ be a dominant map between normal projective varieties. 
    \begin{enumerate}
        \item If $Y$ contains infinitely many ruled divisors then $X$ does.
        \item Suppose moreover that $f$ is generically finite and  $X$ is $\Q$-factorial. If the divisors are movable and big on $Y$ the same holds on $X$.
    \end{enumerate} 
    \end{proposition}
    \begin{proof}
    (1) Not all the infinitely many ruled divisors on $Y$ can be contained in the locus possibly contracted by $f$. Hence the conclusion follows. 

    (2) This item follows immediately from \cite[Lemma 4.3]{LMP21}.
    \end{proof}

Let $X$ be a projective manifold of OG6 type having a divisor $E$ such that $q(E)=-2$ and $div(E)=2$, where $q$ is the Beauville-Bogomolov-Fujiki form.
    By \cite[Prop. 4.3]{Gro22}, $X$ is dominated by an IHS manifold $Y$, which is of $K3^{[3]}$ type and has Picard number at least $17$.

\begin{corollary}
    Let $X$ be a projective manifold of OG6 type having a divisor $E$ such that $q(E)=-2$ and $div(E)=2$. 
    If $Y$ has Picard number strictly smaller that $21$,
    then $X$ contains infinitely many big and movable ruled divisors.
\end{corollary}
\begin{proof}
    Conjecture \ref{conj:weak} holds for $Y$ and by Proposition \ref{prop:dominati} the result follows for $X$. 
\end{proof}

Let $X$ be a projective irreducible symplectic variety contained in tables 4,7 and 8 of \cite{BGMM}.
    By construction, each of them is dominated by a IHS manifolds $Y$ of $K3^{[n]}$ or of Kummer type.

    \begin{corollary}
Let $X$ be a projective irreducible symplectic variety contained in tables 4,7 and 8 of \cite{BGMM}, with the exceptions of the 34 items of tables 7 and 8 with rank 7. If the dominating variety $Y$ has Picard number strictly smaller than the maximal one, then $X$ contains infinitely many big and movable ruled divisors.      
    \end{corollary}
\begin{proof}
If the dominating manifold $Y$ has Picard rank at least two, the claim follows by Theorem \ref{thm:main} applied to $Y$ and Proposition \ref{prop:dominati}. If the action of the group $G$ on $Y$ is non trivial on $H^2$, we have an invariant ample class and all coinvariant algebraic classes, therefore the Picard rank is at least two. This holds precisely for all elements in the statement: the rank of the coinvariant lattice is obtained as the difference between the second Betti number and the rank of the invariant part (contained in the above tables), and is positive unless $Y$ is of Kummer type and the rank is 7. 
\end{proof}

\begin{remark}
    Notice that the existence of infinitely many uniruled divisors for some manifolds of OG10 type was already observed in \cite[Proposition 4.4]{LMP21}.
\end{remark}

 \bibliography{literature}
\bibliographystyle{alpha}
\end{document}